\documentclass[article]{amsproc}
\usepackage{amsfonts}
\usepackage{amssymb}
\usepackage{latexsym}
\usepackage{color}
\usepackage{euscript}  
\numberwithin{equation}{section} \theoremstyle{plain}
\usepackage{units}
\usepackage{euscript}
\usepackage{mathrsfs}
\usepackage{setspace}
\makeatletter
\@namedef{subjclassname@2010}{%
\textup{2010} Mathematics Subject Classification}
\makeatother

\newtheorem{thm}{Theorem}[section]
\newtheorem{prop}[thm]{Proposition}
\newtheorem{cor}[thm]{Corollary}
\newtheorem{lem}[thm]{Lemma}

\newtheorem*{hpt*}{Hipoteza}
\newtheorem*{prob*}{Problem}
\newtheorem*{thm*}{Theorem}
\newtheorem*{pro*}{Proposition}
\newtheorem*{met*}{Method}
\newtheorem*{lem*}{Lemma}

\DeclareMathOperator{\D}{d\!} 
 \DeclareMathOperator{\dom}{D}
 \DeclareMathOperator{\M}{m}

\theoremstyle{definition}

\newtheorem*{exa*}{{\it Example}}

\theoremstyle{remark}
\newtheorem{rem}[thm]{{\it Remark}}
\newtheorem*{rem*}{{\it Remark}}

\def\C{\mathbb{C}}
\def\R{\mathbb{R}}

\def\Z{\mathbb{Z}}

\def\A{{\bf{A}}}

\def\w{{\bf{w}}}

\def\kk{\mathcal K}
\def\ff{\mathcal F}
\def\aa{\mathcal{A}}

\def\bb{\mathcal{B}}

\def\bb{\mathcal{B}}

\def\hh{\mathcal H}

\def\xx{\mathcal X}

\def\N{\mathbb N}

\newcommand{\Le}{\leqslant}
\newcommand{\Ge}{\geqslant}
\newcommand*{\bor}[1]{\mathfrak B(#1)}

\newcommand*{\lin}{\mathtt{lin}\,}

\newcommand*{\escr}{\mathscr E}

\newcommand\ca{C_{\A}}
\newcommand*{\is}[2]{\langle#1,#2\rangle}
\newcommand*{\esf}{\mathsf{E}}
\newcommand*{\hsf}{\mathsf{h}}
\begin{document}
\title[Composition operators via inductive limits]{Unbounded composition operators via inductive limits: cosubnormals with matrical symbols}
\author[P. Budzy\'{n}ski]{Piotr Budzy\'{n}ski}
\address{Katedra Zastosowa\'n Matematyki, Uniwersytet Rolniczy w Krakowie, ul. Balicka 253c, 30-198 Krak\'ow, Poland}
\email{piotr.budzynski@ur.krakow.pl}

\author[P.\ Dymek]{Piotr Dymek}
\address{Katedra Zastosowa\'{n} Matematyki, Uniwersytet Rolniczy w Krakowie, ul.\ Balicka 253c, 30-198 Krak\'ow, Poland}
\email{piotr.dymek@ur.krakow.pl}

\author[A. P{\l}aneta]{Artur P{\l}aneta}
\address{Katedra Zastosowa\'n Matematyki, Uniwersytet Rolniczy w Krakowie, ul. Balicka 253c, 30-198 Krak\'ow, Poland}
\email{artur.planeta@ur.krakow.pl}

\begin{abstract}
We prove, by use of inductive techniques, that assorted unbounded composition operators in $L^2$-spaces with matrical symbols are cosubnormal.
\end{abstract}
\setstretch{1.1}
\maketitle
\section{Introduction}
Composition operators in $L^2$-spaces constitute important class of operators that can be found in many areas of mathematics. They are basic objects in the operatorial model of classical mechanics due to Koopman and von Neumann, ergodic theory, theory of dynamical systems and more. They are also very interesting objects of investigation from the operator theory point of view. They have attracted considerable attention from many mathematicians, which resulted in characterizing many of their properties, mainly in bounded case  (see the monograph \cite{sin-man} and references therein). Unbounded composition operators in $L^2$-spaces have become objects of intensive studies quite recently, but they proved to be extremely interesting (\cite{cam-hor, jab, b-j-j-s-ampa, b-j-j-s-jmaa-14, b-j-j-s-aim, bud-pla-s2}).

Bounded subnormal operators have been introduced by Halmos. Studying subnormality turned out to be highly successful and it led to numerous problems in functional analysis, operator theory and mathematical physics. The theory of bounded operators is well-developed now (see the monograph \cite{con} and references therein). Theory of unbounded subnormals, though having much shorter history, brought plenty of interesting results and problems as well (see \cite{bis,foi,sto-sza-1,sto-sza-2,sto-sza-3} for the foundations). Subnormal operators and their relatives play a vital role in operator theory nowadays.

In this note we deal with assorted composition operators induced by linear transformations of $\R^\kappa$. Such operators have been investigated already in \cite{mla, sto, dan-sto, sto-sto} (in bounded case) and in \cite{b-j-j-s-aim} (in unbounded case). Our main result is a criterion for cosubnormality of these operators (cf. Theorem \ref{matrical-cosubn}). We derive it from a criterion for subnormality given in \cite{b-j-j-s-aim}, for which we provide essentially different proof. Basic ingredients of our approach are inductive limit techniques and a criterion for subnormality of general Hilbert space operators invented in \cite[Theorem 3.1.2]{b-j-j-s-jmaa-12} (which relies heavily on \cite{c-s-sz}). It is known that inductive limits of operators are very useful and versatile tools when dealing with unbounded operators (cf. \cite{mar, jan}). In particular, they can be used when studying the questions of boundedness and dense definiteness of composition operators (cf. \cite{bud-pla-s2}). As we show here, they can be also applied when dealing with cosubnormality.
\section{Preliminaries}
In all what follows $\Z_+$ stands for the set of nonnegative integers and $\N$ for the set of positive integers; $\R$ denotes the set of real numbers, $\C$ denotes the set of complex numbers.

Let $\hh$ and $\{\hh_k\}_{k=1}^\infty$ be Hilbert spaces. If $\hh\subseteq \hh_{k+1}\subseteq \hh_k $ for every $k\in\N$, where ``$\subseteq$'' means inclusion of vector spaces, and $\|f\|_{\hh}=\lim_{k\to\infty}\|f\|_{\hh_k}$ for every $f\in\hh$, then we write $\hh_k\downarrow\hh$ as $k\to\infty$.

Let $\hh$ be a (complex) Hilbert space and $T$ be an operator in $\hh$ (all operators are assumed to be linear in this paper). By $\dom(T)$ we denote the domain of $T$. $\overline{T}$ stands for the closure of $T$, and $T^*$ is the adjoint of $T$ (if it exists). Let $T$ be a closable operator in a complex Hilbert space $\hh$ and $\ff$ be a subspace of $\dom(T)$; if $\overline{T|_\ff}=\overline{T}$, then $\ff$ is said to be a core of $T$. A closed densely defined operator $N$ in $\hh$ is said to be normal if $N^*N=NN^*$. A densely defined operator $S$ in $\hh$ is said to be subnormal if there exists a complex Hilbert space $\kk$ and a normal operator $N$ in $\kk$ such that $\hh \leq \kk$ (isometric embedding) and $Sh = Nh$ for all $h \in \dom(S)$. Finally, a densely defined operator $S$ in $\hh$ is cosubnormal if $S^*$ is subnormal.

Let $(X,\aa,\mu)$ be a $\sigma$-finite measure space. The space of all $\aa$-measurable $\C$-valued functions with $\int|f|^2\D\mu<\infty$ is denoted by $L^2(\mu)=L^2(X,\aa,\mu)$. Let $\A$ be an $\aa$-measurable transformation of $X$, i.e., $\A$ is a self-map of $X$ such that $\A^{-1}(\aa)\subseteq \aa$. Define the measure $\mu\circ \A^{-1}$ on $\aa$ by setting $\mu\circ \A^{-1} (\sigma)=\mu(\A^{-1}(\sigma))$, $\sigma\in\aa$. If $\A$ is nonsingular, i.e., $\mu\circ \A^{-1}$ is absolutely continuous with respect to $\mu$, then the operator
    \begin{align*}
    \ca\colon L^2(\mu) \supseteq \dom(\ca) \to L^2(\mu)
    \end{align*}
given by
   \begin{align*}
    \dom(\ca)=\{f \in L^2(\mu) \colon f\circ \A \in L^2(\mu)\} \text{ and } \ca f=f\circ \A \text{ for } f\in \dom(\ca),
   \end{align*}
is well defined\footnote{Clearly, the reverse is also true, i.e., if $\ca$ is well defined, then $\A$ is nonsingular.} and closed in $L^2(\mu)$ (cf. \cite[Section 3]{b-j-j-s-ampa}). We call it a composition operator (induced by $\A$) and we say that $\A$ is the symbol of $\ca$. If the Radon-Nikodym derivative
    \begin{align*}
    \hsf_{\bf \A}=\frac{\D \mu\circ \A^{-1}}{\D\mu}
    \end{align*}
belongs to $L^\infty(\mu)$, which is the space of all $\C$-valued and essentially bounded functions on $X$, then $\ca$ is bounded on $L^2(\mu)$ and $ \|\ca\|=\|\hsf_{\bf A}\|_{L^\infty(\mu)}^{1/2}$. The reverse is also true. By the measure transport theorem we get
    \begin{align*}
    \dom(\ca)=L^2((1+\hsf_{\bf A})\D\mu).
    \end{align*}
It follows from \cite[Proposition 3.2]{b-j-j-s-ampa} that
    \begin{align}\label{dense}
    \text{$\overline{\dom(\ca)}=L^2(\mu)$ if and only if $\hsf_{\bf A}<\infty$ a.e.\ $[\mu]$}.
    \end{align}
The adjoint of a composition operator induced by $\aa$-bimeasurable transformation turns out to be the weighted composition operators (cf. \cite[Lemma 6.4]{cam-hor} and \cite[Corollary 7.3]{b-j-j-s-ampa}):
\begin{align}\label{adjoint}
\begin{minipage}{75ex}
if $\A$ is an invertible transformation of $X$ such that both the $\A$ and $\A^{-1}$ are $\aa$-measurable and nonsingular, then
$$\dom\big(\ca^*\big)=\Big\{ f\in L^2(\mu)\colon \hsf_{\A}\cdot \big(f\circ \A^{-1}\big)\in L^2(\mu) \Big\}$$
and
$$\ca^* f= \hsf_{\A}\cdot \big(f\circ \A^{-1}\big),\quad f\in\dom\big(\ca^*\big).$$
\end{minipage}
\end{align}

Denote by $\escr_+$ the set of all entire functions $\gamma$ on $\C$ of the form
    $\gamma(z) = \sum_{n=0}^\infty a_n z^n$, for $z \in \C$,
where $a_n$ are nonnegative real numbers and $a_k > 0$ for some $k \Ge 1$. For a given positive integer $\kappa$, a function $\gamma\in\escr_+$ and a norm $|\cdot|$ on $\R^\kappa$ induced by an inner product we define the $\sigma$-finite measure $\mu_\gamma^{|\cdot|}$ on $\bor{\R^\kappa}$, the $\sigma$-algebra of Borel subsets of $\R^\kappa$, by
\begin{align*}
\mu_\gamma^{|\cdot|}(\D x) = \gamma(|x|^2) \M_\kappa(\D x),
\end{align*}
where $\M_\kappa$ is the $\kappa$-dimen\-sional Lebesgue measure on $\R^\kappa$. If $\A$ is a linear transformation of $\R^\kappa$ (clearly, such an $\A$ is $\bor{\R^\kappa}$-measurable), we can verify that the composition operator $C_\A$ in $L^2(\mu_\gamma^{|\cdot|})$ is well-defined if and only if $\A$ is invertible. If this is the case, then (cf.\ \cite[equation (2.1)]{sto})
   \begin{align} \label{pochodna}
    \hsf_\A (x) = \frac{1}{|\det \A|} \frac{\gamma(|\A^{-1}x|^2)}{\gamma(|x|^2)}, \quad x \in \R^\kappa \setminus \{0\}.
   \end{align}
(Here, and later on, $|\det\A|$ stands for the modulus of the determinant of $\A$.) Hence, by \eqref{dense} and \cite[Proposition 6.2]{b-j-j-s-ampa}, each well-defined composition operator $\ca$ is automatically densely defined and injective. The question of boundedness of $\ca$ has the following solution.
\begin{thm}[\mbox{\cite[Proposition 2.2]{sto}}]\label{bounded}
Let $\gamma$ be in $\escr_+$ and $|\cdot|$ be a norm on $\R^\kappa$ induced by an inner product. Let $\A$ be an invertible linear transformation of $\R^\kappa$. Then the following assertions hold$:$
\begin{enumerate}
\item If $\gamma$ is a polynomial, then $\A$ induces bounded composition operator on $L^2(\mu_\gamma^{|\cdot|})$ and on $L^2(\mu_{1/\gamma}^{|\cdot|})$.
\item If $\gamma$ is not a polynomial, then $\A$ induces bounded composition operator on $L^2(\mu_\gamma^{|\cdot|})$ $($resp. on $L^2(\mu_{1/\gamma}^{|\cdot|})$$)$ if and only if $\|\A^{-1}\|\Le 1$ $($resp. $\|\A\|\Le 1$$)$.
\end{enumerate}
\end{thm}
It turns out that subnormality of $\ca$ can be also characterized in terms of the symbol $\A$.
\begin{thm}[\mbox{\cite[Theorem 2.5]{sto}}] \label{matrical-subn-b}
Let $\gamma$ be in $\escr_+$ and $|\cdot|$ be a norm on $\R^\kappa$ induced by an inner product. Let $\A$ be an invertible linear transformation of $\R^\kappa$ such that $\ca$ is a bounded operator on $L^2(\mu_\gamma^{|\cdot|})$. Then $\ca$ is subnormal if and only if $\A$ is normal in $(\R^\kappa, |\cdot|)$.
\end{thm}
We close this section by recalling some information concerning weighted composition operators. Let $(X,\aa, \nu)$ be a $\sigma$-finite measure space, $\A$ be a nonsingular $\aa$-measurable transformation of $X$ and $\w$ be a $\C$-valued $\aa$-measurable mapping on $X$ such the measure $(|\w|^2\D\nu)\circ \A^{-1}$ is absolutely continuous with respect to $\nu$. Weighted composition operator $W_{\A,\w} \colon L^2(\nu) \supseteq \dom(W_{\A,\w}) \to L^2(\nu)$ is defined by
   \begin{align*}
    \dom(W_{\A,\w}) & = \{f \in L^2(\nu) \colon \w \cdot (f\circ \A) \in L^2(\nu)\},\\
    W_{\A,\w} f & = \w \cdot (f\circ \A), \quad f \in \dom(W_{\A,\w}).
   \end{align*}
Any such operator $W_{\A,\w}$ is closed. The operator $W_{\A,\w}$ is densely defined if and only if $\big(\esf_\A (|w|^2)\circ \A^{-1}\big)\hsf_\A<\infty$ a.e. $[\nu]$, where $\esf_\A (\cdot)$ denotes the conditional expectation operator with respect to $\sigma$-algebra $\A^{-1}(\aa)$  (cf. \cite[Lemma 6.1]{cam-hor}; see also \cite{b-j-j-s-wco} for more information concerning unbounded weighted composition operators). In particular, if $\A$ is invertible and $\A^{-1}$ is $\aa$-measurable, then $W_{\A,\w}$ is densely defined if and only if $\big(|w|^2\circ \A^{-1}\big)\cdot \hsf_\A<\infty$ a.e. $[\nu]$.
\section{Criterion for cosubnormality}
Our main result is the following criterion for cosubnormality of unbounded composition operators with matrical symbols in $L^2(\mu_{1/\gamma}^{|\cdot|})$.
\begin{thm} \label{matrical-cosubn}
Let $\gamma$ be in $\escr_+$, $|\cdot|$ be a norm on $\R^\kappa$ induced by an inner product and $\A$ be an invertible linear transformation of $\R^\kappa$. If $\A$ is normal in $(\R^\kappa, |\cdot|)$, then $\ca$ is cosubnormal in $L^2(\mu_{1/\gamma}^{|\cdot|})$.
\end{thm}
The proof of the criterion relies on several results, provided below, which are of independent interest. We begin by proving that certain families generated by characteristic functions attached to $\pi$-systems of sets are dense in $L^2$-spaces.

Recall that a nonempty family $\bb$ of subsets of a given set $X$ is a $\pi$-system, whenever $A \cap B \in \bb$ for all $A$ and $B \in \bb$. In turn, if $\bb$ satisfies:
(a) $\varnothing \in \bb$, (b) $A \in \bb$ $\Longrightarrow$ $X \setminus A \in \bb$, and (c) \big($\{A_i\}_{i=1}^{\infty} \subset \bb $ and $A_i\cap A_j=\varnothing$ for $i\neq j$\big) $\Longrightarrow$ $\bigcup_{i=1}^{\infty} A_i \in \bb$, then $\bb$ is said to be a $\lambda$-system.
\begin{lem}\label{rdzen1+}
Let $(X,\aa,\nu)$ be a measure space. Let $\bb\subseteq\aa$ be family of sets satisfying the following conditions:
\begin{enumerate}
\item[(i)] $\bb$ is a $\pi$-system,
\item[(ii)] $\aa =\sigma(\bb)$, i.e., $\aa$ is generated by $\bb$,
\item[(iii)] there exists $\{X_n\}_{n=1}^\infty\subseteq \bb$ such that $X_n\subseteq X_{n+1}$  and $X=\bigcup_{n=1}^\infty X_n$,
\item[(iv)] $\ff:=\lin\{\chi_\sigma\colon \sigma\in\bb\}$, the linear space spanned by $\{\chi_\sigma\colon \sigma\in\bb\}$, is contained in $L^2(X,\aa,\nu)$.
\end{enumerate}
Then the family $\ff$ is dense in $L^2(X,\aa,\nu)$.
\end{lem}
\begin{proof}
Clearly, by (iii) and (iv), the measure $\nu$ is $\sigma$-finite. For every $k\in\N$, $(X_k, \aa_k, \nu_k)$ is a finite measure space, where $\aa_k=\{\omega\cap X_k\colon \omega\in\aa\}$ and $\nu_k=\nu|_{\aa_k}$, the restriction of $\nu$ to $\aa_k$. For every $k\in\N$ we set $\mathcal{L}_k:=\{\omega\in\aa_k\colon \chi_\omega\in \overline{\ff_k}\}$, where $\overline{\ff_k}$ denotes the $L^2(\nu_k)$-closure of $\ff_k=\lin\{\chi_{\sigma\cap X_k}\colon \sigma\in\bb\}$. Then $\mathcal{L}_k$ is $\lambda$-system and thus, by \cite[Th\'eor\`eme]{sie} (known also as Dynkin's $\pi$-$\lambda$ theorem), we have $\aa_k=X_k\cap\sigma(\bb)\subseteq \mathcal{L}_k$ for all $k\in \N$. Since simple functions are dense in  $L^2$-spaces, we deduce that $\overline{\ff_k}=L^2(\nu_k)$ for every $k\in\N$. This and $\sigma$-finiteness of $\nu$ imply the claim.
\end{proof}
Employing the lemma above and description of the graph norm of $W_{\A,\w}$, we prove that certain families generated by characteristic functions form cores for weighted composition operators.
\begin{prop}\label{rdzen2}
Let $(X,\aa,\nu)$ be a $\sigma$-finite measure space and let $\bb\subseteq\aa$ be a family of sets satisfying conditions {\rm (i)}-{\rm (iii)} of Lemma \ref{rdzen1+}. Let $\A\colon X\to X$ be invertible and such that both $\A$ and $\A^{-1}$ are $\aa$-measurable and nonsingular. Let $\w\colon X\to \C$ be $\aa$-measurable. Assume that $\ff:=\lin\big\{\chi_\sigma\colon \sigma\in\bb\big\}\subseteq \dom(W_{\A,\w})$. Then $\ff$ is a core of $W_{\A,\w}$.
\end{prop}
\begin{proof}
It follows from Lemma \ref{rdzen1+} that $\ff$, and consequently $\dom(W_{\A,\w})$, is dense in $L^2(\nu)$. Thus, by \cite[Lemma 6.1]{cam-hor}, $J:=(|\w|^2\circ A^{-1})\cdot \hsf_\A<\infty$ a.e. $[\nu]$. This in turn implies that the measure $J\D\nu$ is $\sigma$-finite. Now, by the measure transport theorem, we have
\begin{align*}
\|f\|^2 + \|W_{\A,\w} f \|^2 &= \|f\|^2 + \int \w^2 \cdot |f|^2 \circ \A \D \nu \\
&= \|f\|^2 + \int \w^2 \circ \A^{-1} \cdot |f|^2 \cdot\hsf_\A \D \nu  = \int |f|^2 (1 + J) \D \nu .
\end{align*}
Thus $\ff$ is a core of $W_{\A, \w}$ if and only if $\ff$ is dense in $L^2\big((1+J)\D\nu\big)$. Since the measure $(1+J)\D\nu$ is $\sigma$-finite, it suffices to apply Lemma \ref{rdzen1+} (with $(1+J)\D\nu$ in place of $\nu$) to prove the claim.
\end{proof}
As a consequence we get the following (cf.\ \cite[Theorem 4.7]{b-j-j-s-ampa}).
\begin{cor}\label{core}
Let $\gamma$ be in $\escr_+$, $|\cdot|$ be a norm on $\R^\kappa$ induced by an inner product, $\A$ be an invertible linear transformation of $\R^\kappa$ and $\ca$ be the composition operator in $L^2(\mu_\gamma^{|\cdot|})$ induced by $\A$. Then $\dom^\infty(\ca)$ is a core of $\ca$.
\end{cor}
\begin{proof}
Let $\bb$ denote the family of all sets of the form $\sigma\cap\{x\in\R^\kappa\colon |x|\leq k\}$ with $\sigma\in\bor{\R^\kappa}$ and $k\in\N$. Then $\bb$ satisfies conditions (i)-(iv) of Lemma \ref{rdzen1+}. Moreover, $\ff:=\lin \{\chi_\omega\colon \omega\in \bb\}\subseteq \dom^\infty(C_\A)$. This and Proposition \ref{rdzen2} implies that $\ff$ and consequently $\dom^\infty(\ca)$ are cores of $\ca$.
\end{proof}
\begin{rem}
Another way of proving that $\dom^\infty(\ca)$ is a core of $\ca$ is to use the so-called Mittag-Leffler theorem, as it was done in the proof of \cite[Theorem 4.7]{b-j-j-s-ampa}).
\end{rem}
Cosubnormality of a composition operator induced by linear transformation of $\R^\kappa$ is strongly related to subnormality of a composition operator induced by the inverted symbol. This is a consequence of the following fact, which essentially is due to Stochel (see \cite[equality (UE) on page 309]{sto} for the case of bounded operators).
\begin{lem}\label{uniteqiv}
Let $\gamma$ be in $\escr_+$, $|\cdot|$ be a norm on $\R^\kappa$ induced by an inner product and $\A$ be an invertible linear transformation of $\R^\kappa$. Then the operators $|\det\A|\,\ca^*$ in $L^2(\mu_{1/\gamma}^{|\cdot|})$ and $C_{\A^{-1}}$ in $L^2(\mu_{\gamma}^{|\cdot|})$ are unitarily equivalent.
\end{lem}
\begin{proof}
Clearly, the map $U\colon L^2(\mu_{1/\gamma}^{|\cdot|})\ni f\mapsto f_\gamma (\cdot):=\frac{f(\cdot)}{\gamma(|\cdot|^2)}\in L^2(\mu_\gamma^{|\cdot|})$ is a unitary operator. By \eqref{adjoint} and \eqref{pochodna}, $f\in L^2(\mu_{1/\gamma}^{|\cdot|})$ belongs to $\dom(\ca^*)$ if and only if
\begin{align*}
\int_{\R^\kappa} \big|(f\circ A^{-1})(x)\hsf_\A(x)\big|^2 \D\mu_{1/\gamma}^{|\cdot|}(x)<\infty.
\end{align*}
Since, by the change-of-variable theorem (cf. \cite[Theorem 7.26]{rud}), we have
\allowdisplaybreaks
\begin{align*}
\int_{\R^\kappa} |f_\gamma(x)|^2 \hsf_{\A^{-1}}(x)\D \mu_\gamma^{|\cdot|}(x)
&=\int_{\R^\kappa} |f(\A^{-1}x)|^2 \frac{\gamma(|x|^2)}{\big(\gamma(|\A^{-1}x|^2)\big)^2}\D \M_\kappa(x)\\
&=|\det \A|^2\int_{\R^\kappa} \big|(f\circ A^{-1})(x)\hsf_\A(x)\big|^2 \D\mu_{1/\gamma}^{|\cdot|}(x),
\end{align*}
we see that $f\in \dom(\ca^*)$ is equivalent to $f_\gamma\in \dom(C_{\A^{-1}})$. This and elementary computations implies that $C_{\A^{-1}}U=U |\det \A| \ca^*$, which proves our claim.
\end{proof}
It was shown in \cite[Theorem 32]{b-j-j-s-aim} that a normal linear transformation $\A$ of $(\R^\kappa, |\cdot|)$ induces subnormal composition operator $\ca$ in $L^2(\mu_\gamma^{|\cdot|})$. The proof of this fact involved a highly non-trivial construction of a measurable family of probability measures satisfying the so-called consistency condition. Below we prove this fact in a different manner, based on the following version of \cite[Theorem 3.1.2]{b-j-j-s-jmaa-12} (we include the proof, which is similar to that of the original result, for the reader's convenience).
\begin{lem}\label{projective}
Let $S$ be a densely defined operator in a complex Hilbert space $\hh$. Suppose that there are a family $\{\hh_k\}_{k\in\N}$ of Hilbert spaces such that $\hh_k\downarrow\hh$ as $k\to\infty$, and a set $\xx\subseteq \hh$ such that
   \begin{enumerate}
   \item[(i)] $\xx \subseteq \dom^\infty(S)$,
   \item[(ii)] $\ff:=\lin \bigcup_{n=0}^\infty S^n(\xx)$ is a core of $S$,
   \item[(iii)] $\ff$ is dense in $\hh_k$ for every $k \in \N$,
   \item[(iv)] $S|_{\ff}$ is  a subnormal operator in $\hh_k$ for every $k \in \N$.
   \end{enumerate}
Then $S$ is subnormal.
\end{lem}
\begin{proof}
We prove that $S|_\ff$ is subnormal in $\hh$. To this end we consider any finite system $\{a_{p,q}^{i,j}\}_{p,q = 0, \ldots, n}^{i,j=1, \ldots, m} \subset \C$ such that
\begin{align} \label{1}
\sum_{i,j=1}^m \sum_{p,q=0}^n a_{p,q}^{i,j} \lambda^p \bar \lambda^q z_i \bar z_j \Ge 0, \quad \lambda, z_1, \ldots, z_m \in \C,
\end{align}
Since $S|_\ff$ is subnormal in $\hh_k$ for every $k\in\N$ and $\ff$ is invariant for $S$, we obtain by \cite[Theorem 21]{c-s-sz} that
\begin{align*}
\sum_{i,j=1}^m \sum_{p,q=0}^n a_{p,q}^{i,j} \is{S^p f_i} {S^q f_j}_{\hh_k} \Ge 0, \quad f_1, \ldots, f_m \in \ff,\ k\in\N.
\end{align*}
Clearly, the polarization formula and the fact that $\hh_k\downarrow \hh$ as $k\to\infty$ imply that $\lim_{k\to\infty} \is{x}{y}_{\hh_k}=\is{x}{y}_\hh$ for all $x,y\in\hh$. Therefore we have
\begin{align*}
\sum_{i,j=1}^m \sum_{p,q=0}^n a_{p,q}^{i,j} \is{S^p f_i} {S^q f_j}_{\hh} \Ge 0, \quad f_1, \ldots, f_m \in \ff.
\end{align*}
In view of \cite[Theorem 21]{c-s-sz}, the above implies subnormality of $S|_\ff$ in $\hh$. This and the fact that $\ff$ is a core of $S$ yields subnormality of $S$.
\end{proof}

\begin{prop}[\mbox{\cite[Theorem 32]{b-j-j-s-aim}}] \label{matrical-subn}
Let $\gamma$ be in $\escr_+$, $|\cdot|$ be a norm on $\R^\kappa$ induced by an inner product and $\A$ be an invertible linear transformation of $\R^\kappa$. If $\A$ is normal in $(\R^\kappa, |\cdot|)$, then $\ca$ is subnormal in $L^2(\mu_\gamma^{|\cdot|})$.
\end{prop}
\begin{proof}
Since $\gamma\in\escr_+$, we have $\gamma(z) = \sum_{n=0}^\infty a_n z^n$ for all $z \in \C$. For $k\in\N$, let $\gamma_k$ be a polynomial given by $\gamma_k(z)= \sum_{n=0}^k a_n z^n$, $z\in\C$. Without loss of generality we may assume that $a_1>0$. Hence for every $k\in\N$, $\mu_{\gamma_{k}}\neq 0$. Since for every $l\in\N$, $\mu_{\gamma_l}^{|\cdot|}$ is absolutely continuous with respect to $\mu_{\gamma_{l+1}}^{|\cdot|}$, and $\mu_{\gamma_{l}}^{|\cdot|}$ is absolutely continuous with respect to $\mu_{\gamma}^{|\cdot|}$, we deduce that $L^2(\mu_{\gamma_k}^{|\cdot|})\downarrow L^2(\mu_\gamma^{|\cdot|})$ as $k\to\infty$. Moreover, by Theorems \ref{bounded} and \ref{matrical-subn-b}, $\A$ induces a bounded subnormal composition operator on every $L^2(\mu_{\gamma_k}^{|\cdot|})$, $k\in\N$. Let $\xx=\dom^\infty(\ca)$. Then we have
\begin{align*}
\ff:=\lin \bigcup_{n=0}^\infty \ca^n(\xx)= \dom^\infty(\ca)
\end{align*}
and so, by Corollary \ref{core}, $\ff$ is a core of $\ca$. Moreover, for every $k\in\N$, $\ca|_{\ff}$ is subnormal in $L^2(\mu_{\gamma_k}^{|\cdot|})$. We complete the proof by applying Lemma \ref{projective}.
\end{proof}
Combining Lemma \ref{uniteqiv} and the proposition above we may prove Theorem \ref{matrical-cosubn}.
\begin{proof}[Proof of Theorem \ref{matrical-cosubn}]
Since $\A$ is normal in $(\R^\kappa, |\cdot|)$, $\A^{-1}$ is normal in $(\R^\kappa, |\cdot|)$ as well. This implies that $C_{\A^{-1}}$ is subnormal in $L^2(\mu_\gamma^{|\cdot|})$ by Proposition \ref{matrical-subn}. Therefore, by Lemma \ref{uniteqiv}, we see that $\ca$ is cosubnormal in $L^2(\mu_{1/\gamma}^{|\cdot|})$.
\end{proof}


\begin{thebibliography}{99}
   \bibitem{bis}            E. Bishop, Spectral theory for operators on a Banach space,
                            {\em Trans. Amer. Math. Soc.} {\bf 86} (1957), 414-445.
\bibitem{b-j-j-s-jmaa-12}   P. Budzy\'nski, Z. J. Jab{\l}o\'nski, I. B. Jung, J. Stochel, Unbounded subnormal weighted shifts on directed trees,
                            {\em J. Math. Anal. Appl.} {\bf 394} (2012), 819-834.
\bibitem{b-j-j-s-ampa}      P. Budzy\'nski, Z. J. Jab{\l}o\'nski, I. B. Jung, J. Stochel, On unbounded composition operators in $L^2$-spaces,
                            {\em Ann. Mat. Pura Appl.} {\bf 193} (2014), 663-688.
\bibitem{b-j-j-s-jmaa-14}   P. Budzy\'nski, Z. J. Jab{\l}o\'nski, I. B. Jung, J. Stochel, A multiplicative property characterizes quasinormal composition operators in $L^2$-spaces,
                            {\em J. Math. Anal. Appl.} {\bf 409} (2014), 576-581.
\bibitem{b-j-j-s-aim}       P. Budzy\'nski, Z. J. Jab{\l}o\'nski, I. B. Jung, J. Stochel, Unbounded subnormal composition operators in $L^2$-spaces,
                            {\em J. Funct. Anal.}, to appear.
\bibitem{b-j-j-s-wco}       P. Budzy\'nski, Z. J. Jab{\l}o\'nski, I. B. Jung, J. Stochel, Unbounded weighted composition operators in $L^2$-spaces,
                            {\em preprint}, http://arxiv.org/abs/1310.3542.
\bibitem{bud-pla-s2}        P. Budzy\'nski, A. P{\l}aneta, Dense definiteness and boundedness of composition operators in $L^2$-spaces via inductive limits,
                            {\em submitted for publication}, http://arxiv.org/abs/1409.3961.
\bibitem{cam-hor}           J. T. Campbell, W. E. Hornor, Seminormal composition operators.
                            {\em J. Operator Theory} {\bf 29} (1993), 323-343.
\bibitem{c-s-sz}             D. Cicho\'n, J. Stochel, F. H. Szafraniec, Extending positive definiteness,
                            {\em Trans. Amer. Math. Soc.} {\bf 363} (2011), 545-577.
\bibitem{con}               J. B. Conway, The theory of subnormal operators,
                            Mathematical Surveys and Monographs, {\bf 36}, American Mathematical Society, Providence, RI, 1991.
\bibitem{dan-sto}           A. Daniluk, J. Stochel, Seminormal composition operators induced by affine transformations,
                            {\em Hokkaido Math. J.} {\bf 26} (1997), 377-404.
\bibitem{foi}               C. Foia\c{s}, D\'ecompositions en op\'erateurs et vecteurs propres. I., \'Etudes de ces d\`ecompositions et leurs rapports avec les prolongements des op\'erateurs,
                            {\em Rev. Roumaine Math. Pures Appl.} {\bf 7} (1962), 241-282.
\bibitem{jab}               Z. Jab{\l}o\'{n}ski, Hyperexpansive composition operators,
                            {\em Math. Proc. Camb. Phil. Soc.} {\bf 135} (2003), 513-526.
\bibitem{jan}               J. Janas, Inductive limit of operators and its applications,
                            {\em Studia Math.} {\bf 90} (1988), 87-102.
\bibitem{mla}               W. Mlak, Operators induced by transformations of {G}aussian variables,
                            {\em Ann. Polon. Math.} \textbf{46} (1985), 197-212.
\bibitem{mar}               A. V. Marchenko, Selfadjoint differential operators with an infinite number of independent variables,
                            {\em Mat. Sb. $($N.S.$)$}, {\bf 96}, (1975), 276–293.
\bibitem{rud}               W. Rudin, Real and Complex Analysis,
                            McGraw-Hill, 1987.
\bibitem{sie}               W. Sierpi\'nski, Un th\'eor\`{e}me g\'en\'erale sur les families d'ensemble,
                            {\em Fundamenta Mathematicae} {\bf 12} (1928), 206-210
\bibitem{sin-man}           R. K. Singh, J. S. Manhas, Composition Operators on Function Spaces,
                            North-Holland, 1993.
\bibitem{sto}               J. Stochel, Seminormal composition operators on ${L}^2$ spaces induced by matrices,
                            {\em Hokkaido Math. J.} \textbf{19} (1990), 307--324.
\bibitem{sto-sto}           J. Stochel, J. B. Stochel, Seminormal composition operators on ${L}^2$ spaces induced by matrices: {T}he {L}aplace density case,
                            {\em J. Math. Anal. Appl.} \textbf{375} (2011), 1-7.
\bibitem{sto-sza-1}         J. Stochel, F. H. Szafraniec, On normal extensions of unbounded operators. I,
                            {\em J. Operator Theory} {\bf 14} (1985), 31-55.
\bibitem{sto-sza-2}         J. Stochel and F. H. Szafraniec, On normal extensions of unbounded operators, II,
                            {\em Acta Sci. Math. $($Szeged$)$} {\bf 53} (1989), 153-177.
\bibitem{sto-sza-3}         J. Stochel, F. H. Szafraniec, On normal extensions of unbounded operators. III. Spectral properties,
                            {\em Publ. RIMS, Kyoto Univ.} {\bf 25} (1989), 105-139.
\end{thebibliography}
\end{document}